\documentclass[11pt, reqno]{amsart}
\usepackage{charter}
\usepackage{mathrsfs}

\usepackage{graphicx} 
\usepackage[margin=1in]{geometry}
\usepackage{amsmath,amsthm,amssymb,enumerate,fancyhdr,bbm}
\usepackage{natbib}
\usepackage{hyperref} 
\usepackage{mathtools}
\usepackage{tikz, enumerate, multicol}
\usepackage{amsthm}
\usepackage{color}
\usepackage{graphicx}
\usepackage{subfig}
\usepackage{grffile}
\usepackage{tikz-network}
\usetikzlibrary{graphs}
\usetikzlibrary{automata, positioning, math}
\usetikzlibrary{shapes.geometric, arrows, shapes.misc, arrows.meta}
\tikzset{
cnd/.style={
draw, circle, minimum size=10pt, inner sep=0pt, font=\tiny
}}

\definecolor{vrdg}{RGB}{0,0,0}
\newcommand{\vrdg}[1]{{\color{vrdg}{#1}}}
\hypersetup{colorlinks=true, linkcolor=blue, citecolor=blue, filecolor=blue, urlcolor=blue}
\DeclareMathOperator{\one}{\mathbbm{1}} 

\numberwithin{equation}{section}
\newtheorem{theorem}{Theorem}[section]

\newtheorem{lemma}{Lemma}[section]

\newtheorem{remark}{Remark}[section]

\usepackage{mathtools}

\usepackage{xcolor}

\newcommand{\epsn}{\varepsilon_n}

\begin{document}
\title{The Generalized Friendship Paradox for Eigenvectors}

\author[B. Bhattacharya]{Bishakh Bhattacharya} 
\address{Indian Statistical Institute, 203, B.T. Road, Kolkata.} 
\email{bishakh.rik@gmail.com}
\author[A. Chakrabarty]{Arijit Chakrabarty}
\address{Indian Statistical Institute, 203, B.T. Road, Kolkata.}
\email{arijit.isi@gmail.com}
\author[R. S. Hazra]{Rajat Subhra Hazra}
\address{Mathematical Institute, Leiden University, Einsteinweg 55, 2333 CC Leiden, The Netherlands}
\email{r.s.hazra@math.leidenuniv.nl}

\date{\today}

\begin{abstract}
In this paper, we investigate the generalized friendship paradox for eigenvectors (alternatively called the eigen friendship paradox and abbreviated hereafter as EFP) in the setting of inhomogeneous Erd\H{o}s--R\'enyi random graphs whose edge probabilities are generated by a continuous graphon. We consider the adjacency matrix of the graph and take the entries of the eigenvector corresponding to its largest eigenvalue as the vertex attributes. It was shown in \cite{hazra2026generalized} that the generalized friendship paradox holds in this setting. We study the empirical distribution of the resulting bias values across the vertices and derive its limiting distribution explicitly in terms of the principal eigenvalue and the corresponding eigenfunction of the integral operator whose kernel is the underlying graphon. 
\end{abstract}

\keywords{Graphs \& Random Graphs, Generalized Friendship Paradox, Degree, Weak Convergence in probability, Eigenvalues, Eigenvectors, Graphons, Integral Operators}
\subjclass[2000]{05C80, 60B20, 60B10, 46L54}
\maketitle

\section{Introduction}
“Your friends have more friends than you do” is a phenomenon known in the literature as the friendship paradox. It was first observed by the American sociologist Scott L. Feld (\cite{feld1991your}). Despite its seemingly paradoxical nature, it has since been rigorously established as a mathematical fact (\cite{HdHP2023,cantwell2021friendship}). Broadly speaking, it asserts that in a network of mutual friendships, the average difference between the number of friends of an individual's friends and the individual's own number of friends is always non-negative.

Over the years, the friendship paradox has found numerous applications. In \cite{christakis2010social}, it was employed for the early detection of contagious disease outbreaks by monitoring the friends of randomly sampled individuals, who tend to occupy more central positions in the network and are therefore more likely to become infected at an earlier stage of an outbreak. In \cite{nettasinghe2019your}, the paradox was used to develop an efficient methodology for predicting election outcomes based on opinion polls. In \cite{christakis2008collective}, the authors analyzed social networks consisting of distinct clusters of smokers and non-smokers and found that smoking cessation is better understood as a collective phenomenon rather than an isolated individual decision. A probabilistic perspective on the friendship paradox was developed in \cite{cao2016friendship}, where it was shown that a randomly selected friend of an individual stochastically has more friends than the individual themselves.

Studies on the \emph{Generalized Friendship Paradox} (GFP), which concerns vertex attributes more general than degrees (such as wedges, triangles, and other local graph statistics), have also appeared in the literature. In \cite{cantwell2021friendship,bhattacharya2025triangle}, it was shown that the GFP holds whenever the attribute under consideration is positively correlated with a function of the vertex degrees. In \cite{eom2014generalized}, the GFP was investigated in the context of the scholarly output of researchers using co-authorship networks derived from \emph{Physical Review} journals together with data from Google Scholar. The authors found that the GFP holds for attributes such as publications, collaborations, and citations, indicating that, on average, a researcher's co-authors outperform the researcher with respect to these measures.

In \cite{hodas2013friendship}, the authors put forward the provocative claim, reflected in the title of their paper, that “your friends are more interesting than you are.” Using data from the social network X (formerly Twitter), they supported this claim through a threefold argument: first, that users' friends and followers tend to have more friends and followers than the users themselves (the classical friendship paradox); second, that they are more active; and third, that they generate and receive more viral content. The latter two phenomena are referred to therein as the \emph{Activity Paradox} and the \emph{Virality Paradox}, respectively.

The GFP does not hold universally. Examples illustrating its failure are provided in \cite{bhattacharya2025triangle}, where the paradox was investigated for both deterministic and random graphs with vertex triangle counts as the underlying attribute. The authors showed that the GFP holds for \emph{partially completed star graphs}, but may fail in other cases. In particular, failure of the GFP is exhibited in certain random graph models generated by two-block graphons. 

\cite{HdHP2023} studied the empirical distribution of friendship biases for different sparse random graph models. In \cite{hazra2026multi}, this study was extended to a multi-level setting, where a person's "friends" are replaced by vertices reached after $k$ steps of either a backtracking or non-backtracking random walk on a sparse random graph. The empirical distribution of these $k$-level friendship biases for locally tree-like random graphs are studied here. The authors prove convergence of the empirical distribution as the graph size $n\to\infty$, and /or the exploration depth $k\to\infty$, and also when $k$ grows with $n$, resulting in commuting limits in $n$ and $k$ in the non-back tracking case while the same isn't true in the backtracking exploration--depends on the mixing time of the exploration. \cite{hazra2025friendship} studies the friendship paradox on trees, where vertices are classified according to whether their neighbors have higher or lower average degree. It proves that, in every finite tree, positive vertices are at least as numerous as negative ones. The authors also analyze infinite Galton–Watson trees, deriving the limiting proportions of each vertex type. The results show that the friendship paradox can behave quite differently in random infinite trees depending on the offspring distribution.

Centrality measures have also gained significant attention in the last few decades. The analyses in \cite{page1999pagerank} explored PageRank as a centrality measure, showing its importance in ranking web pages and also its effectiveness in browsing and estimating traffic on the web network. In \cite{dasaratha2020distributions}, several convergence results for Katz--Bonacich and eigenvector centralities were established in the setting of large random networks generated from stochastic block models that are not excessively sparse. In \cite{avella2018centrality}, a number of centrality measures, including degree, eigenvector, Katz, and PageRank centralities, were formulated in the graphon framework, and their asymptotic properties were investigated. More recently, \cite{hazra2026generalized} introduced an analogue of the GFP, termed the \emph{centrality paradox}, and showed that degree, eigenvector, Katz, and PageRank centralities all satisfy this paradox.

The present work brings together the themes of the generalized friendship paradox, eigenvector centrality, and the \emph{centrality paradox} in the setting of inhomogeneous Erd\H{o}s--R\'enyi random graphs. We consider random graphs whose edge probabilities are determined by a symmetric graphon and study the GFP when the attribute assigned to each vertex is the corresponding entry of the principal eigenvector of the adjacency matrix. Equivalently, this amounts to studying the centrality paradox for the eigenvector centrality. We show that the paradox indeed holds in this setting. Our main focus is the asymptotic behavior of the empirical distribution of the resulting bias values. In the context of the classical friendship paradox, as stated previously, \cite{HdHP2023} conducted an extensive study of this distribution using the framework of local weak convergence of random graphs and further investigated several properties of the limiting measure for models such as the Erd\H{o}s--R\'enyi graph, the configuration model, and the preferential attachment model. In contrast, we work in the graphon-based setting and analyze the dense and sparse regimes separately. We show that the empirical distribution of the bias values converges weakly in probability to a deterministic probability measure. Moreover, this limiting measure admits an explicit characterization: it is the law of a random variable that can be expressed in terms of the principal eigenvalue and the corresponding eigenfunction of the integral operator associated with the underlying graphon, along with a $U(0,1)$ random variable.

\section{Setup and Main Results}
\label{ss.theorems}


\paragraph{ \bf Eigenvector Friendship Paradox.}
Let $G_n$ be a finite undirected graph with $n$ vertices labelled by $[n] = \{1,\ldots,n\}$. Let $\Delta^E_{i,n}$ denote the \emph{eigen friendship-bias (EFB) of vertex $i \in [n]$} given by
\begin{equation}
\Delta^E_{i, n} =  \left[\frac{1}{d_i} \sum_{j \in [n]} A_{ij} r_j^{(n)} - r_i^{(n)}\right] 1_{\{d_i \neq 0\}},
\label{eq:efpbias}
\end{equation}
where $(A_{ij})_{i,j \in [n]}$ denotes the $(i,j)$-th entry of the adjacency matrix $A_n$, $d_i$ is the degree of vertex $i \in [n]$, and let $r = r^{(n)} = (r_i^{(n)}) \ge 0$ be the eigenvector corresponding to the largest eigenvalue $\lambda_1$ of the adjacency matrix $A_n$. That is, $A_nr^{(n)} = \lambda_1r^{(n)}$ and $\frac{1}{n}\sum_{i=1}^n \left(r_i^{(n)}\right)^2 =1$. The \emph{eigen-friendship bias}  is defined as
\begin{equation}
\Delta_{[n]}^E = \frac{1}{n} \sum_{i \in [n]} \Delta_{i,n}^E.
\label{EFBDef}
\end{equation}
Let $D$ be the diagonal matrix with entries as degrees, and define $C$ as 
\[
C_{ij} = \frac{A_{ij}}{d_i}\one_{\{d_i >0\}}.
\]
If the graph is connected, we interpret $ C$ as $ D^ {-1}A$. In what follows, we are primarily concerned with connected graphs. Consequently $\lambda_1 >0$. Observe that \eqref{eq:efpbias} implies 
\[
\Delta^E_{i, n} = (Cr^{(n)})_i - r_i^{(n)}.
\]
As shown in \cite{hazra2026generalized}, we have      
\begin{align}
\label{EFPDeterministic}\langle 1,Cr^{(n)}\rangle \ge \langle 1,r^{(n)}\rangle \quad \Leftrightarrow \quad \frac{1}{n} \sum_{i \in [n]} \Delta_{i,n}^E \ge0.
\end{align}
That is, the eigen friendship paradox holds. For completion, we provide a proof of \eqref{EFPDeterministic} in Lemma \ref{EFP}. 
Suppose the graph is connected. Using the fact that $$A_nr^{(n)}=\lambda_1 r^{(n)}$$ we get
\begin{align}
\nonumber \Delta^E_{i, n} &= (Cr^{(n)})_i - r_i^{(n)},\\
\nonumber & = (D^{-1}Ar^{(n)})_i - r_i^{(n)},\\
\nonumber & = \frac{(Ar^{(n)})_i}{d_i}-r_i^{(n)},\\
\label{EFBmodified}& = \frac{\lambda_1}{d_i}r_i^{(n)} - r_i^{(n)}.
\end{align}
We define the empirical distribution of the eigen-friendship bias as 
\begin{align}
\label{ESDEFP} \mu_n^{E} = \frac{1}{n}\sum_{i=1}^n \delta_{\Delta_{i,n}^E}.   
\end{align}
Let us consider a continuous graphon $W:[0,1]^2\to [0,1]$ such that $W(x,y)=W(y,x)$ and $W \not\equiv 0$. We will present our study based on two graphon models, which are successively stated below.\\
\textbf{Model A: } We consider the inhomogeneous Erd\H{o}s-R\'enyi random graph with edge probabilities given by $\mathbb{P}(A_{ij}=1)=W(\frac{i}{n},\frac{j}{n})$.

\medskip\noindent
\textbf{Model B:} Let $X_1,X_2, X_3,\ldots$ be a sequence of i.i.d. $U(0,1)$ variables. Let $\mathbb{X}$ denote the random sample $(X_1,X_2, X_3,\ldots,X_n)$ and let  $X_{(i)}$ denote the corresponding $i$-th order statistic. We now consider the Erd\H{o}s-R\'enyi random graph with edge probabilities of the form $\mathbb{P}(A_{ij} =1 \mid \mathbb{X})=\varepsilon_n W\left(X_{(i)},X_{(j)}\right)$, i.e. conditioned on $\mathbb{X}$,
\[
A_{ij} \overset{indpt}{\sim} Ber\left(\varepsilon_n W\left(X_{(i)},X_{(j)}\right)\right),
\]
where, for some $\xi > 4,$ fixed throughout,
\[
\lim_{n \to \infty} \frac{(\log n)^\xi}{n\varepsilon_n} =0.
\]
In both models the above edge probabilities are used for $i\neq j$, and we set $A_{ii}=0\, \forall\, i=1,2,\ldots,n$. We define the degree functional as 
\[
K(x)=\int_0^1W(x,y)dy.
\]
Further, let us consider the integral operator, defined $T_W:L^2[0,1]\to L^2[0,1],$ as
\begin{align}
\label{LimitIntergralOperatorDef}(T_Wf)(x)=\int_0^1W(x,y)f(y)dy.
\end{align}
The following technical assumptions are needed.\\
\textbf{Assumption A1} $K(x) \ge k_{min} >0 .$\\
\textbf{Assumption A2} Let $\rho = \lambda_1(T_W)$ be simple and isolated, with $\mathrm{dist}(\rho, \sigma(T_W)\setminus \rho)\ge \gamma>0$.\\
\textbf{Assumption A3} The eigenfunction corresponding to $\rho$, denoted by $\phi,$ satisfies $\phi > 0$ and $\|\phi\|_{L^2}=1$.
\begin{remark}
The assumption about the spectral gap in \textbf{A2} is natural; e.g., it holds when $W > 0$ almost everywhere then the Krein-Rutman Theorem gives a simple positive principal eigenvalue. The quantitative gap is then positive because $T_W$ is compact and the principal eigenvalue is isolated. 
\end{remark}
With all of the above assumptions in place, we are now in a position to state our main theorem. We first state a fundamental lemma about the \emph{ EFP}.
\begin{lemma}\label{EFP}({\bf Generalized Friendship Paradox for EigenVectors})
For the adjacency matrix $A_n$ of a connected undirected graph with atleast two vertices,
\begin{align*}
\Delta_{[n]}^E\ge 0\quad\text{i.e.}\,   \langle 1,Cr^{(n)}\rangle \ge \langle 1,r^{(n)}\rangle,
\end{align*}with equality if and only if the graph is regular.
\end{lemma}
\begin{theorem}\label{Th.ESDEFB}({\bf Weak Convergence of the Empirical eigen-friendship bias distribution})
Consider either Model A or Model B and suppose the Assumptions A1--A3 hold. For the empirical eigen-friendship bias distribution defined in \eqref{ESDEFP} we have, as $n \to  \infty$
\[
\mu_n^E \Rightarrow \mu_{\infty} \quad \text{in probability,}
\] where $\mu_{\infty}$ is the law of 
\[
\Delta_{\infty} = \left(\frac{\rho}{K(U)}-1\right)\phi(U),
\]where $U \sim U(0,1).$
\end{theorem}

Since bounded Lipschitz functions form a convergence-determining class for weak convergence on $\mathbb R$, it is enough to show that, for every bounded Lipschitz function $F$,
\begin{align}
\label{TargetESDEFB}\frac{1}{n}\sum_{i=1}^n F(\Delta_{i,n}^E) \overset{p}{\to} \int_0^1 F\left(\left(\frac{\rho}{K(x)}-1\right)\phi(x)\right)dx.
\end{align}
We first prove the stronger $L^1[0,1]$ convergence of the bias step functions. Define
\[
\Delta_n^I(x)=\sum_{i=1}^n \Delta_{i,n}^E\mathbf 1_{[(i-1)/n,i/n)}(x).
\]
It is enough to prove
\begin{align}
\label{ESDinprobinL1}\int_0^1 \left|\Delta_n^I(x)-\left(\frac{\rho}{K(x)}-1\right)\phi(x)\right|dx \overset{p}{\to}0.
\end{align}
For $x\in(0,1]$ set $r_n(x)=r_{\lceil nx\rceil}^{(n)}$; the value at $x=0$ is irrelevant. Then
\begin{equation}
\label{EFB_Idendityfn} \Delta_n^I(x)=\left(\frac{\lambda_1}{d_{\lceil nx\rceil}}-1\right)r_n(x).
\end{equation}
We observe that
\begin{align}
&\int_0^1\left|\left(\frac{\lambda_1}{d_{\lceil nx\rceil}}-1\right)r_n(x)-\left(\frac{\rho}{K(x)}-1\right)\phi(x)\right|dx\nonumber\\
&\quad\le \left\|\frac{\lambda_1}{d_{\lceil n\cdot\rceil}}-\frac{\rho}{K(\cdot)}\right\|_2+
\left\|\frac{\rho}{K(\cdot)}-1\right\|_2\|r_n-\phi\|_2.
\label{MainFormulation}
\end{align}
Here we used Cauchy--Schwarz and the normalization $n^{-1}\sum_{i=1}^n(r_i^{(n)})^2=1$. Thus it remains to prove
\[
 \textbf{Part 1:}\, \, \int_0^1\left|\frac{\lambda_1}{d_{\lceil nx\rceil}}-\frac{\rho}{K(x)}\right|^2dx\overset{p}{\to}0,   
\]
and
\[
\textbf{Part 2:} \,\, \int_0^1|r_n(x)-\phi(x)|^2dx\overset{p}{\to}0. 
\]
The same formulation applies to Model B after using the normalization $n \varepsilon_n$.

\begin{remark}
It is worth noting that for the desired weak convergence in the above theorem, the set of bounded Lipschitz functions on $\mathbb{R}$ also acts as a convergence-determining class. 
\end{remark}
\section{Proofs}
To begin with, we present a proof of Lemma \ref{EFP} which says that the generalized friendship paradox for eigenvectors is true.
\begin{proof}[Proof of Lemma \ref{EFP}:]
Recalling the definition of $EFB$ in \eqref{EFBDef} and equation \eqref{EFBmodified}, we have 
\[
\Delta_{[n]}^E = \frac{1}{n}\sum_{i=1}^n \left(\frac{\lambda_1}{d_i}r_i^{(n)} - r_i^{(n)}\right).
\]Thus, it is sufficient to show that
\begin{align}
\label{EFPTarget}\lambda_1\sum_{i=1}^n \frac{r_i^{(n)}}{d_i} \ge \sum_{i=1}^n r_i^{(n)}.
\end{align}Denoting the vector $d = \left(d_1,d_2,\ldots,d_n\right)$, note that $A_n1 = d.$ This  along with the symmetry of $A_n$ implies
\begin{align}
\label{EFPInterm1}\sum_{i=1}^n r_i^{(n)} d_i = \langle r,d\rangle=\langle r,A_n1\rangle = \langle A_nr,1\rangle = \lambda_1 \sum_{i=1}^n r_i^{(n)}.
\end{align}
By Perron-Frobenius, $r_i^{(n)}>0$ for every $i$. Let us consider the weights $w_i = \frac{r_i^{(n)}}{\sum_{i=1}^n r_i^{(n)}}$. Clearly $w_i \ge 0\, \forall\,i \ge 1$ and $\sum_{i=1}^n w_i=1.$ Thus by the weighted harmonic–arithmetic mean inequality in \cite{maze2009note} we get,
\begin{align*}
&\left(\sum_{i=1}^n w_id_i\right)\left(\sum_{i=1}^n \frac{w_i}{d_i}\right) \ge 1,\\
\implies & \left(\frac{\sum_{i=1}^n r_i^{(n)}d_i}{\sum_{i=1}^n r_i^{(n)}} \right) \left(\frac{\sum_{i=1}^n \frac{r_i^{(n)}}{d
_i}}{\sum_{i=1}^n r_i^{(n)}}\right) \ge 1.
\end{align*}
Now invoking \eqref{EFPInterm1} into the first factor of the LHS above, we get \eqref{EFPTarget}. Lastly, equality in the harmonic--arithmetic mean inequality holds if and only if all $d_i$ are equal, which is precisely regularity. This completes the proof.
\end{proof}
\textbf{ Proof of Theorem \ref{Th.ESDEFB} when $F(x)=x$ :} Though from \eqref{MainFormulation} and the observations thereafter, we know that \textbf{Part 1} and \textbf{Part 2} are the main points to be covered, there are several other technical steps that have to be passed. As the proof of Theorem \ref{Th.ESDEFB} constitutes the main result of this paper, for the reader's convenience, we begin by providing a brief outline of the proof. 
\begin{enumerate}
\item We first show that the maximum absolute deviation of the degrees upon scaling by $n$ (respectively $n\varepsilon_n$ in the sparse case) converges to $0$ in probability.
\item We use the first step to show that the function $d_{\lceil n.\rceil}$ upon scaling by $n$ (respectively $n\varepsilon_n$ in the sparse case) converges to $K$ in probability in $L^1.$
\item Next, we show that the difference of the spectral norms of $A_n$ and the mean matrix (conditional mean in the sparse case) is small with high probability.
\item We then show that the largest eigenvalue upon scaling by $n$ (respectively $n \varepsilon_n$ in the sparse case) converges to $\rho$ in probability.
\item We prove \textbf{Part 1}.
\item We prove \textbf{Part 2}.
\end{enumerate}We now proceed towards proving the above steps. 

\begin{lemma}\label{DegMaxDev}
As $n \to \infty$, for \textbf{Model A} we have 
\begin{align*}
\max_{1 \le i \le n}\left|\frac{d_i}{n}-K\left(\frac{i}{n}\right)\right|\overset{p}{\to}0,    
\end{align*} while for \textbf{Model B} we have,
\begin{align*}
    \max_{1 \le i \le n}\left|\frac{d_i}{n\varepsilon_n}-K\left(\frac{i}{n}\right)\right|\overset{p}{\to}0.
\end{align*}
\end{lemma}
\begin{proof}
For \textbf{Model A}, it suffices to show the following:
\begin{align*}
\max_{1 \le i \le n}\left|\frac{d_i}{n}-\mathbb{E}\left(\frac{d_i}{n}\right)\right|\overset{p}{\to}0\quad \text{and}\quad \max_{1 \le i \le n}\left|\mathbb{E}\left(\frac{d_i}{n}\right)-K\left(\frac{i}{n}\right)\right|\to 0.
\end{align*}
To check the first one, note that for $\epsilon >0,$
\begin{align*}
\mathbb{P}\left(\max_{1 \le i \le n}\left|\frac{d_i}{n}-\mathbb{E}\left(\frac{d_i}{n}\right)\right| > \epsilon\right) &\le \sum_{i=1}^n \mathbb{P}\left(\left|d_i-\mathbb{E}\left(d_i\right)\right| > n\epsilon\right),\\
& = \sum_{i=1}^n\mathbb{P}\left(\left|\sum_{j=1}^n\left(A_{ij}-\mathbb{E}(A_{ij})\right)\right|>n\epsilon\right).
\end{align*}
Note that as $W$ maps to $[0,1]$, $\left|K(x)\right|\le 1$ for all $x \in [0,1]$, and also $\left|A_{ij}-\mathbb{E}(A_{ij})\right| \le 1,\forall\, i,j.$
Also $$\mathbb{E}\left(A_{ij}-\mathbb{E}(A_{ij})\right)^2= W\left(\frac{i}{n},\frac{j}{n}\right)\left(1-W\left(\frac{i}{n},\frac{j}{n}\right)\right)\le 1/4.$$
Thus, by Bernstein's inequality, we get,
\begin{align*}
& \mathbb{P}\left(\left|\sum_{j=1}^n\left(A_{ij}-\mathbb{E}(A_{ij})\right)\right|>n\epsilon\right)\\ 
&\le 2\exp\left(-\frac{\frac{n^2\epsilon^2}{2}}{\sum_{j=1}^n\left[W\left(\frac{i}{n},\frac{j}{n}\right)\left(1-W\left(\frac{i}{n},\frac{j}{n}\right)\right)\right]+\frac{2n\epsilon}{3}} \right),\\
& \le 2\exp\left(-\frac{n^2\epsilon^2}{\frac{n}{2}+ \frac{4n\epsilon}{3}} \right).
\end{align*}
Hence, using a union bound, we get
\begin{align*}
\mathbb{P}\left(\max_{1 \le i \le n}\left|\frac{d_i}{n}-\mathbb{E}\left(\frac{d_i}{n}\right)\right| > \epsilon\right) &\le 2\sum_{i=1}^n \exp\left(-\frac{n^2\epsilon^2}{\frac{n}{2}+ \frac{4n\epsilon}{3}} \right),\\
&= 2n\exp{\left(-\frac{n^2\epsilon^2}{\frac{n}{2}+ \frac{4n\epsilon}{3}} \right)} \to 0\quad \text{as n}\to \infty.
\end{align*}
For the remainder, note that
\begin{align}
\nonumber \max_{1 \le i \le n} \left|\mathbb{E}\left( \frac{d_i}{n}\right) - K\left(\frac{i}{n}\right)\right| &= \max_{1 \le i \le n} \left|\frac{1}{n}\sum_{j=1}^n W\left(\frac{i}{n},\frac{j}{n}\right) - K\left(\frac{i}{n}\right)\right|,\\
\label{EFPAStep1Part2.1}& \le \sup_{x \in [0,1]}\left|\frac{1}{n}\sum_{j=1}^n W\left(x,\frac{j}{n}\right) - K\left(x\right)\right|.
\end{align}So it is enough to show that the RHS above goes to zero.
Also as the empirical measure $\lambda_n:= \frac{1}{n}\sum_{i=1}^n \delta_{\frac{i}{n}}$ weakly converges to the Lebesgue measure on [0,1], we have for each fixed $x$,
\begin{align}
\label{ESDStep1Part2Interm1}\frac{1}{n}\sum_{j=1}^n W\left(x,\frac{j}{n}\right) = \int_0^1W(x,y)\lambda_n(dy) \to \int_0^1W(x,y)dy=K(x).
\end{align}
Now, for $\epsilon>0,$ denoting by $\omega$ the modulus of continuity of $W$, we can find $\eta>0$ such that $\omega(\eta) < \epsilon.$ Using the compactness of [0,1], we find a finite $\eta$-net, $\{x_1,x_2,\ldots x_l\}$, say.
Then for any $x \in [0,1],$ we can find $x_k$ from the $\eta$-net such that $\left|x-x_k\right|< \eta.$ Then,
\begin{align*}
& \left|\frac{\sum_{j=1}^n W\left(x,\frac{j}{n}\right)}{n} -K(x)\right|,\\
\le & \left|\frac{\sum_{j=1}^n \left(W\left(x,\frac{j}{n}\right)-W\left(x_k,\frac{j}{n}\right)\right)}{n} \right| + \left|\frac{\sum_{j=1}^n W\left(x_k,\frac{j}{n}\right)}{n}-K(x_k)\right|+|K(x_k) - K(x)|,\\
\le & 2\epsilon + \max_k\left|\frac{\sum_{j=1}^n W\left(x_k,\frac{j}{n}\right)}{n}-K(x_k)\right|.
\end{align*}Using \eqref{ESDStep1Part2Interm1} and the fact that we are dealing with finitely many $k's$ and since $\epsilon$ is arbitrary, we get as $n \to \infty,$
\begin{align*}
\sup_{x\in [0,1]}\left|\frac{\sum_{j=1}^n W\left(x,\frac{j}{n}\right)}{n} -K(x)\right| \to 0.
\end{align*}This completes the proof.
For the \textbf{Model B} we will need to tweak the above argument slightly due to the randomized graphon arguments. Here instead of two-term decomposition we will need a three-term decomposition, and then we will conclude using a standard $3\epsilon$ argument. We will show that
\[
\max_{1 \le i \le n} \left|\frac{d_i}{n\varepsilon_n} - \mathbb{E}\left( \frac{d_i}{n\varepsilon_n} \mid \mathbb{X}\right)\right| \overset{p}{\to} 0,\quad \max_{1 \le i \le n} \left|\mathbb{E}\left( \frac{d_i}{n\varepsilon_n} \mid \mathbb{X}\right) - K\left(X_{(i)}\right)\right| \overset{p}{\to} 0,
\]and,
\[
\max_{1 \le i \le n }\left|K\left(X_{(i)}\right) - K\left(\frac{i}{n}\right)\right| \overset{p}{\to} 0.
\]
Note that as conditioned on $\mathbb{X}$, $A_{ij} \overset{indpt}{\sim} Ber\left(\varepsilon_n W\left(X_{(i)}, X_{(j)}\right)\right),$ by a similar argument using Bernstein's inequality as before we get for $\delta > 0,$  
\begin{align*}
&\mathbb{P}\left(\left| d_i - \mathbb{E}(d_i|\mathbb{X})\right| > n\varepsilon_n\delta | \mathbb{X}\right) \le 2\exp\left(-\frac{n^2\varepsilon_n^2\delta^2}{\frac{n\varepsilon_n}{2}+ \frac{4n\varepsilon_n\delta}{3}} \right).
\end{align*}Therefore, using a union bound again, we get as $n \to \infty$,
\begin{align}
\mathbb{P}\left(\max_{1 \le i \le n}\left|d_i - \mathbb{E}(d_i|\mathbb{X})\right| > n\varepsilon_n \delta\right)
\nonumber \le & \sum_{i=1}^n \mathbb{E}\,\mathbb{P}\left(|d_i -\mathbb{E}(d_i|\mathbb{X}) | > n\varepsilon_n\delta|\mathbb{X}\right)\\
\label{EFPBStep1Part1}\le & 2n\exp\left(-\frac{n^2\varepsilon_n^2\delta^2}{\frac{n\varepsilon_n}{2}+ \frac{4n\varepsilon_n\delta}{3}} \right) \to 0.
\end{align}
Next, similar to \eqref{EFPAStep1Part2.1} we get
\begin{align*}
\max_{1 \le i \le n} \left|\mathbb{E}\left( \frac{d_i}{n\varepsilon_n} | \mathbb{X}\right) - K\left(X_{(i)}\right)\right| \le \sup_{x \in [0,1]}\left|\frac{1}{n}\sum_{j=1}^n W\left(x,X_{(j)}\right) - K\left(x\right)\right|.
\end{align*}We will show that RHS above goes to $0$ almost surely. As before, using the modulus of continuity $\omega$ of $W$ and the compactness of [0,1], for any $\epsilon >0,$ we can find $\delta >0$ such that $\omega(\delta) < \epsilon$ and also a $\delta$-net, $\{x_1,x_2,\ldots, x_l\},$ so that for any $x\in [0,1]$ and we can find $x_k$ such that $|x-x_k|< \delta.$Then using a similar argument as in \textbf{Model A} we get
\begin{align*}
\sup_{x \in [0,1]}\left|\frac{1}{n}\sum_{j=1}^n W\left(x,X_{(j)}\right) -K(x)\right| & \le 
 2\omega(\delta) + \max_{k=1,2,\ldots,l}\left|\frac{1}{n}\sum_{j=1}^n W\left(x_k,X_{(j)}\right) - K(x_k)\right|.
\end{align*}
Now, for any such k we have by classical SLLN, as $n \to \infty,$ \begin{align*}
\frac{1}{n}\sum_{j=1}^n W\left(x_k,X_{(j)}\right) - K(x_k) & = \frac{1}{n}\sum_{j=1}^n W\left(x_k,X_j\right) - K(x_k),\\
& \overset{a.s.}{\to } \mathbb{E}\left(W(x_k,X_1)\right) - K(x_k) = K(x_k)-K(x_k)=0. 
\end{align*}
Since we have only finitely many such $k$'s, we get, as $n \to \infty,$
\[
\max_{k=1,2,\ldots,l} \left|\frac{1}{n}\sum_{j=1}^n W\left(x_k,X_{(j)}\right) - K(x_k)\right| \overset{a.s.}{\to }0,
\]and thus, as $n \to \infty,$
\begin{align}
\label{EFPBStep1Part2}  \sup_{x \in [0,1]}\left|\frac{1}{n}\sum_{j=1}^n W\left(x,X_{(j)}\right) - K\left(x\right)\right| \overset{a.s.}{\to }0.  
\end{align}
Lastly, we consider the term  $\max_i \left|K(X_{(i)}) - K(\frac{i}{n})\right|$. Denoting the sample cdf of the $X_i's$ by $F_n(.)$ and the cdf of $U(0,1)$ by $F,$ we get using the Glivenko-Cantelli theorem that,
\begin{align}
\label{EFPBStep1OrderStats}\max_i\left|X_{(i)}-\frac{i}{n}\right| & = \max_i\left|F_n(X_{(i)})-F(X_{(i)})\right|\le \sup_{x \in [0,1]}\left|F_n(x)-F(x)\right|\overset{a.s.}{\to} 0.
\end{align}
As $K$ is uniformly continuous on [0,1], it admits a modulus of continuity, $\omega_K,$ say. and hence we get using the non-decreasing nature of $\omega_K,$
\begin{align}
\label{EFPBStep1Part3}\max_i\left|K(X_{(i)}) - K\left(\frac{i}{n}\right)\right| \le \omega_K\left(\max_i\left|X_{(i)}-\frac{i}{n}\right|\right) \overset{a.s.}{\to} 0,
\end{align}where the last implication follows by the fact that $\omega_K(x) \to 0 $ as $x \to 0.$

Finally, a $3\epsilon$-argument involving the outcomes of equations \eqref{EFPBStep1Part1},\eqref{EFPBStep1Part2} and \eqref{EFPBStep1Part3} together leads to,
\[
\max_i\left|\frac{d_i}{n\varepsilon_n}-K\left(\frac{i}{n}\right)\right| \overset{p}{\to }0.
\]
\end{proof}
\begin{remark}\label{DegDevL1}
 Lemma \ref{DegMaxDev} entails as $n \to \infty$, 
\[
\int_0^1\left|\frac{d_{\lceil nx \rceil}}{n} -K(x)\right|\overset{p}{\to} 0.
\]To see this note that,
\begin{align}
\nonumber \int_0^1\left|\frac{d_{\lceil nx \rceil}}{n} -K(x)\right|& \le \int_0^1\left|\frac{d_{\lceil nx \rceil}}{n} -K\left(\frac{\lceil nx \rceil}{n}\right)\right|dx + \int_0^1 \left|K\left(\frac{\lceil nx \rceil}{n}\right)- K(x)\right|dx,\\
\label{EFPStep4}&\le \max_i\left|\frac{d_i}{n} - K\left(\frac{i}{n}\right)\right|+ \int_0^1 \left|K\left(\frac{\lceil nx \rceil}{n}\right)- K(x)\right|dx.      
\end{align} 
Clearly then Lemma \ref{DegMaxDev} and the continuity of $K$ and DCT  ensure that the RHS of \eqref{EFPStep4} converges to $0$ in probability, as desired.
For \textbf{Model B}, the exact same argument with $n$ replaced by $n\varepsilon$ in the denominator of the first term in the integrand gives, 
\[
\int_0^1\left|\frac{d_{\lceil nx \rceil}}{n\varepsilon} -K(x)\right|\overset{p}{\to} 0.
\]

\end{remark}

\begin{lemma}\label{ConcentrationAroundMean}
We denote the mean matrix $\mathbb{E}A_n$ by $P_n$. So, 
\[
\text{for \bf Model A, } \left(P_n\right)_{ij} = W\left(\frac{i}{n},\frac{j}{n}\right)\quad \text{and for \bf Model B, }\left(P_n\right)_{ij} = \varepsilon_n W\left(X_{(i)},X_{(j)}\right).
\]
Then as $n \to \infty,$ for \textbf{Model A} we have
\[
\left\|\frac{A_n}{n}-\frac{P_n}{n}\right\|\overset{p}{\to} 0,
\]
and for \textbf{Model B} we have,
\[
\left\|\frac{A_n}{n\varepsilon_n}-\frac{P_n}{n\varepsilon_n}\right\|\overset{p}{\to} 0.
\]
\end{lemma}
\begin{proof}
For \textbf{Model A} we first show that, 
\[
\mathbb{P}\left(\|A_n-P_n\|> 2\sqrt{Mn} + C_1n^{1/4}(\log n)^{\xi/4}\right) = O\left(n\exp(-(\log n)^{\xi/4})\right),
\]
where $C_1$ is a suitable positive constant, $M=1/4$ and $\xi>4.$
In view of the proof of Lemma 4.1 from \cite{chakrabarty2020eigenvalues}, we get
\begin{align}
\label{ExTr} \mathbb{E}\left[\mathrm{Tr}(A_n-P_n)^k\right] \le k_1n\left(2\left(\sqrt{Mn}\right)^k\right),
\end{align}
where $k_1$ is a positive constant and there exists a constant $a>0$ such that the exponent $k \in 2\mathbb{N}$ can be chosen as $$k=\sqrt{2}a(Mn)^{1/4}.$$
Using the fact that $(1-x)^k \le \exp(-kx),$ for $k,x>0,$ we get by applying Markov's inequality and \eqref{ExTr},
\begin{align*}
\mathbb{P}\left(\left\|A_n-P_n\right\|> 2\sqrt{Mn} + C_1n^{1/4}(\log n)^{\xi/4}\right) &\le \frac{\mathbb{E}\|A_n-P_n\|^k}{\left(2\sqrt{Mn} + C_1n^{1/4}(\log n)^{\xi/4}\right)^k},\\
&\le \frac{\mathbb{E}\mathrm{Tr}(A_n-P_n)^k}{\left(2\sqrt{Mn} + C_1n^{1/4}(\log n)^{\xi/4}\right)^k},\\
&\le \frac{k_1n\left(2\sqrt{Mn}\right)^k}{\left(2\sqrt{Mn} + C_1n^{1/4}(\log n)^{\xi/4}\right)^k},\\
& = k_1 n\left(1-\frac{C_1n^{1/4}(\log n)^{\xi/4}}{\left(2\sqrt{Mn} + C_1n^{1/4}(\log n)^{\xi/4}\right)}\right)^k,\\ 
&\le k_1n\exp\left(-\frac{kC_1n^{1/4}\left(\log n\right)^{\xi/4}}{2\sqrt{Mn} + C_1n^{1/4}(\log n)^{\xi/4}}\right).
\end{align*}
Now invoking the above value of $k$ and using the fact that $(\log n)^\xi \ll n$, a regular computation along the lines of Lemma 4.1 from \cite{chakrabarty2020eigenvalues} leads to the,
\begin{align*}
\mathbb{P}\left(\left\|A_n-P_n\right\|> 2\sqrt{Mn} + C_1n^{1/4}(\log n)^{\xi/4}\right) &\le  
k_1n\exp\left(-\frac{\sqrt2a C_1M^{1/4}(\log n)^{\xi/4}}{2\sqrt{M}+C_1C_2}\right),\\
&= O\left(n\exp\left(-(\log n)^{\xi/4}\right)\right),
\end{align*}
for all large $n$, where $C_2 >0$ is such that $\frac{(\log n)^\xi}{n} < C_2$ for all $n \ge 1.$
This proves the claim and hence, 
\[
\|A_n-P_n\|=O_p(\sqrt{n})=o_p(n),
\]which completes the proof for \textbf{Model A}.\\ For \textbf{Model B}, we show a similar high (conditional) probability bound:
\[
\mathbb{P}\left(\|A_n-P_n\| > 2\sqrt{n\varepsilon_n} + C_1(n\varepsilon_n)^{1/4}(\log n)^{\xi/4}\right|\mathbb{X})=O\left(n\exp\left(-(\log n)^{\xi/4}\right)\right), 
\]where $C_1$ is a suitable positive constant and $\xi>4.$ Going along the lines of the proof of Lemma 4.1 from \cite{chakrabarty2020eigenvalues} again, we get
\begin{align}
\label{ExTrConditional} \mathbb{E}\left[\mathrm{Tr}(A_n-P_n)^k|\mathbb{X}\right] \le k_1n\left(2\left(\sqrt{n\varepsilon_n}\right)^k\right),
\end{align}
where $k_1$ is a positive constant and there exists a constant $a>0$ such that the exponent $k \in 2\mathbb{N}$ can be chosen as $$k=\sqrt{2}a(n\varepsilon_n)^{1/4}.$$
Using the fact that $(1-x)^k \le \exp\{-kx\},$ for $k,x>0,$ we get by applying the Conditional Markov inequality (Exercise 34.9 in \cite{billingsley2017probability} ) and \eqref{ExTrConditional},
\begin{align*}
&\mathbb{P}\left(\left\|A_n-P_n\right\|> 2\sqrt{n\varepsilon_n} + C_1(n\varepsilon_n)^{1/4}(\log n)^{\xi/4}|\mathbb{X}\right)\\& \le k_1n\exp\left(-\frac{kC_1(n\varepsilon_n)^{1/4}\left(\log n\right)^{\xi/4}}{2\sqrt{n\varepsilon_n} + C_1(n\varepsilon_n)^{1/4}(\log n)^{\xi/4}}\right).
\end{align*}
As done for \textbf{Model A}, invoking the above value of $k$ and using the sparsity assumption that $(\log n)^\xi \ll n\varepsilon_n$, a similar computation  along the lines of the proof of Lemma 4.1 in \cite{chakrabarty2020eigenvalues} leads to
\begin{align*}
\mathbb{P}\left(\left\|A_n-P_n\right\|> 2\sqrt{n\varepsilon_n} + C_1(n\varepsilon_n)^{1/4}(\log n)^{\xi/4}|\mathbb{X}\right) & \le k_1n\exp\left(-\frac{\sqrt2a C_1(\log n)^{\xi/4}}{2+C_1C_2}\right),\\ 
&= O\left(n\exp\left(-(\log n)^{\xi/4}\right)\right),    
\end{align*} for all large $n$, where $C_2 >0$ is such that $\frac{(\log n)^\xi}{n\varepsilon_n} < C_2 \forall n \ge 1.$ 
Therefore we get as $n \to \infty$,
\begin{align}
\nonumber &\mathbb{P}\left(\left\|A_n-P_n\right\|> 2\sqrt{n\varepsilon_n} + C_1(n\varepsilon_n)^{1/4}(\log n)^{\xi/4}\right),\\
\nonumber =&\mathbb{E}\left(\mathbb{P}\left(\left\|A_n-P_n\right\|> 2\sqrt{n\varepsilon_n} + C_1(n\varepsilon_n)^{1/4}(\log n)^{\xi/4}|\mathbb{X}\right)\right),\\
\label{EFPStep2}=&  O\left(n\exp\left(-(\log n)^{\xi/4}\right)\right)\to 0.
\end{align}
This proves the desired high probability bound for \textbf{Model B} and hence, 
\[
\|A_n-P_n\|=O_p(\sqrt{n\varepsilon_n})=o_p(n\varepsilon_n),
\]
which completes the proof.
\end{proof}

\begin{remark}
Denoting the largest eigenvalue of $P_n$ by $\lambda_1(P_n),$ it immediately follows from Lemma \ref{ConcentrationAroundMean} and Weyl's inequality that for \textbf{Model A},
\[
\left|\lambda_1 - \lambda_1(P_n)\right|\le \|A_n-P_n\|=o_p(n),
\]and for \textbf{Model B},
\[
\left|\lambda_1 - \lambda_1(P_n)\right|\le \|A_n-P_n\|=o_p(n\varepsilon).
\]
\end{remark}
\begin{lemma}\label{LargestEigenvalueConvergence}
As $n \to \infty$, for \textbf{Model A} we have,
\[
\frac{\lambda_1}{n} \overset{p}{\to} \rho,
\]and for \textbf{Model B} we have,
\[
\frac{\lambda_1}{n\varepsilon_n} \overset{p}{\to} \rho,
\]where $\rho$ is the largest eigenvalue of the integral operator defined in  \eqref{LimitIntergralOperatorDef}.
\end{lemma}

\begin{proof}
For \textbf{Model A}, let $T_{W_n}$ be the integral operator on $L^2(0,1)$ with kernel,
$$W_n(x,y) = W\left(\frac{\lceil nx\rceil}{n},\frac{\lceil ny \rceil}{n}\right).$$
We will first show that the non-zero spectra of $T_{W_n}$ and $P_n/n$ are identical and then  as $n \to \infty$,
\[
\|T_{W_n} - T_W\|_{op} \to 0,
\]from which the lemma will follow. Note that, as $W \not\equiv 0$ and continuous,  $\lambda_1(P_n/n), \lambda_1(T_{W_n})$ and $\lambda_1(T_W)$ are all $>0$.
We define two operators 
$$S_n^1 : \{x_k\}_{k=1}^{n} \mapsto \sqrt{n}\sum_{k=1}^{n}x_kI_{I_k},$$ and
$$S_n^2: f\mapsto \left\{\sqrt{n}\int_{I_k}f(x)dx\right\}_{k=1}^{n},$$ where
$I_k=[\frac{k-1}{n},\frac{k}{n}],k=1,2,\ldots n.$
Then, as done in \cite{bottcher2007norms}, we get 
\begin{align}
\label{ForSpectrumEquality2}S_n^1P_nS_n^2 = nT_{W_n}.
\end{align} 
Let $\alpha \neq 0$ be an eigenvalue of $S_n^1P_nS_n^2$, and let $y$ be a corresponding eigenvector. Then
\[
S_n^1P_nS_n^2 y=\alpha y .
\]
Applying \(S_n^2\) to both sides gives
\[
S_n^2S_n^1P_nS_n^2 y=\alpha S_n^2y .
\]
Since \(S_n^2S_n^1\) is the identity operator on \(\mathbb C^n\), it follows that
\[
P_nS_n^2y=\alpha S_n^2y .
\]
Moreover, \(S_n^2y\neq 0\), because otherwise the last display would imply
\[
0=\alpha S_n^2y=S_n^2S_n^1P_nS_n^2y,
\]
and, since \(\alpha\neq0\), this would contradict \(y\neq0\) in the original eigenvalue equation. Hence \(\alpha\) is an eigenvalue of \(P_n\), with eigenvector \(S_n^2y\).
Thus any eigenvalue of $S_n^1P_nS_n^2$ is an eigenvalue of $P_n$.
Now suppose $\beta$ is an eigenvalue of $P_n$ with eigenvector $z.$ Then, note that,
\begin{align*}
(S_n^1P_nS_n^2)(S_n^1z) &= S_n^1P_nS_n^2S_n^1z,\\
& = S_n^1P_nz = \beta S_n^1z.
\end{align*}Thus $\beta$ is an eigenvalue of $S_n^1P_nS_n^2$ with eigenvector $S_n^1z.$ 
Equation \eqref{ForSpectrumEquality2} now implies that the non-zero spectra are the same and consequently, the largest eigenvalues of $T_{W_n}$ and $P_n/n$ are identical. An application of Weyl's inequality and Lemma \ref{ConcentrationAroundMean} gives
\begin{align}
\nonumber \left|\frac{\lambda_1}{n} - \lambda_1(T_{W_n}) \right|
&\le \left|\frac{\lambda_1}{n} - \frac{\lambda_1(P_n)}{n} \right| + \left|\frac{\lambda_1(P_n)}{n} - \lambda_1(T_{W_n}) \right|,\\
\nonumber & = \frac{1}{n}\left|\lambda_1 - \lambda_1(P_n) \right|,\\
\nonumber & \le \frac{1}{n}\|A_n-P_n\|,\\
\label{EFPAStep3Part1}& = O_p\left(\frac{1}{\sqrt{n}}\right).
\end{align}
We now look at $\left|\lambda_1(T_{W_n}) - \lambda_1(T_W)\right| = \left|\lambda_1(T_{W_n}) - \rho\right|$. As integral operators with symmetric kernels are self-adjoint and compact, again by Weyl's inequality, we get 
\[
\left|\lambda_1(T_{W_n}) - \lambda_1(T_W)\right| \le \left\|T_{W_n} - T_W\right\|_{op}.
\]We will show that the $\left |W\left(\frac{\lceil nx \rceil}{n},\frac{\lceil ny \rceil}{n}\right) -W(x,y)\right|$ converges to zero uniformly, which will imply that the RHS above goes to zero.
Recalling that the modulus of continuity $\omega$ of $W$ is non-decreasing and using the facts that $\|.\|_{l_2} \le \|.\|_{l_1}$ we get 
\begin{align}
\nonumber\left |W\left(\frac{\lceil nx \rceil}{n},\frac{\lceil ny \rceil}{n}\right) -W(x,y)\right| &\le \omega\left(\left\|\left(\frac{\lceil nx \rceil}{n},\frac{\lceil ny \rceil}{n}\right) - (x,y)\right\|_{l_2}\right),\\
\nonumber& \le \omega\left(\left\|\left(\frac{\lceil nx \rceil}{n},\frac{\lceil ny \rceil}{n}\right) - (x,y)\right\|_{l_1}\right),\\
\nonumber& = \omega\left( \left|\frac{\lceil nx \rceil}{n} -x\right| + \left|\frac{\lceil ny \rceil}{n} -y\right|\right),\\
\label{EFPAStep3Part2}& \le \omega\left(\frac{2}{n}\right).
\end{align}
Since the above is true $\forall (x,y) \in [0,1]^2$ we get the desired uniform convergence. Now, equations \eqref{EFPAStep3Part1} and \eqref{EFPAStep3Part2} together complete the proof for \textbf{Model A}. \\
For \textbf{Model B} , due to the additional randomization in the setup, we consider a sequence of random operators $T_{W_n}$(with a little abuse of notation we continue using the same symbols as in \textbf{Model A}) with kernel,
$$W_n(x,y) =  W\left(X_{(\lceil nx \rceil)},X_{(\lceil ny\rceil)}\right).$$ Since we are dealing with random operators, we fix a sample point here. Going along the arguments for \textbf{Model A}, it follows that the non-zero spectra of $T_{W_n}$ and $P_n/n$ are identical. Consequently,
\[
\frac{\lambda_1(P_n)}{n\varepsilon} = \frac{\lambda_1(T_{W_n})}{\varepsilon}.
\]
Again as done in achieving \eqref{EFPAStep3Part1} in the proof of \textbf{Model A}  by Weyl's inequality and Lemma \ref{ConcentrationAroundMean}, we get
\begin{align}
\label{EFPBStep3Part1} \left|\frac{\lambda_1}{n\varepsilon_n} - \lambda_1(T_{W_n}) \right|
 & \le \frac{1}{n\varepsilon_n}\left|\lambda_1 - \lambda_1(P_n) \right|
\le \frac{1}{n\varepsilon_n}\|A_n-P_n\|
 = O_p\left(\frac{1}{\sqrt{n\varepsilon_n}}\right).
\end{align}
We now look at $\left|{\lambda_1(T_{W_n})} - \lambda_1(T_W)\right| = \left|\frac{\lambda_1(T_{W_n})}{\varepsilon} - \rho\right|$. Again along the lines of the proof of \textbf{Model A} we have by Weyl's inequality, 
\[
\left|\frac{\lambda_1(T_{W_n})}{\varepsilon} - \lambda_1(T_W)\right| \le \left\|{T_{W_n}} - T_W\right\|_{op}.
\]
Using the modulus of continuity $\omega$ of $W$ again along with the fact that $\|.\|_{l_2} \le \|.\|_{l_1}$ we get as in \textbf{Model A},
\begin{align*}
&\left |W\left(X_{(\lceil nx \rceil)},X_{(\lceil ny\rceil)}\right) -W(x,y)\right|,\\
& \le \omega\left(\left|X_{(\lceil nx \rceil)}-\frac{\lceil nx \rceil}{n}\right| + \left|X_{(\lceil ny \rceil)}-\frac{\lceil ny \rceil}{n}\right| + \left|\frac{\lceil nx \rceil}{n} -x\right| + \left|\frac{\lceil ny \rceil}{n} -y\right|\right),\\
& \le \omega\left(2\max_i\left|X_{(i)}-\frac{i}{n}\right| + \frac{2}{n}\right).
\end{align*}
Since the above is true $\forall (x,y) \in [0,1]^2$, using \eqref{EFPBStep1OrderStats} we get as $n \to \infty,$ 
\begin{align}
\label{EFPBStep3Part2}\left\|{T_{W_n}} - T_W\right\|_{op}^2 \le \|W_n-W\|_\infty^2 \le \omega^2\left(2\max_i\left|X_{(i)}-\frac{i}{n}\right| + \frac{2}{n}\right) \overset{a.s.}{\to}0. 
\end{align}
Now, equations \eqref{EFPBStep3Part1} and \eqref{EFPBStep3Part2} together complete the proof for \textbf{Model B}. \\
\end{proof}
We now proceed towards the penultimate step of the proof.
\begin{lemma}\label{Part1}[Proof of \bf Part 1]
As $n \to \infty$, we have 
\[
\int_0^1 \left|\frac{\lambda_1}{d_{\lceil nx\rceil}}-\frac{\rho}{K(x)}\right|^2dx \overset{p}{\to} 0.
\]
\end{lemma}
\begin{proof}
We break down the integral in the usual way as follows:
\begin{align*}
\int_0^1 \left|\frac{\lambda_1}{d_{\lceil nx\rceil}}-\frac{\rho}{K(x)}\right|^2dx = \int_0^1 \frac{\lambda_1^2}{d_{\lceil nx\rceil}^2}dx - 2\lambda_1\rho\int_0^1\frac{1}{d_{\lceil nx\rceil}K(x)}dx+ \rho^2\int_0^1\frac{1}{K^2(x)}dx.
\end{align*}
As before, we first look at \textbf{Model A}. We define the event $E_n$ as 
\[
E_n: = \left\{\min_i \frac{d_i}{n} > \frac{K_{min}}{2}\right\}.
\]
Let $\tilde{i} = \operatorname{arg}\min_i d_i.$ Then on $E_n^c$ we have 
\begin{align*}
\max_i\left|\frac{d_i}{n}-K\left(\frac{i}{n}\right)\right| &\ge \left|\frac{d_{\tilde{i}}}{n}-K\left(\frac{\tilde{i}}{n}\right)\right|,\\
& \ge K\left(\frac{\tilde{i}}{n}\right)-\frac{d_{\tilde{i}}}{n},\\
& > K_{min} - \frac{K_{min}}{2}=\frac{K_{min}}{2}.
\end{align*}Hence, by Lemma \ref{DegMaxDev} we obtain as $n \to \infty$,
\[
\mathbb{P}(E_n^c) \to 0.
\]Now we note that as $\frac{d_{\lceil nx\rceil}}{n},K(.) \le 1$, on $E_n$ we have using assumption $A1$,
\begin{align*}
\left|\frac{n^2}{d_{\lceil nx\rceil}^2} - \frac{1}{K^2(x)}\right|
&= \left|\frac{\left(K(x)+\frac{d_{\lceil nx\rceil}}{n}\right)\left(K(x)-\frac{d_{\lceil nx\rceil}}{n}\right)}{\left(\frac{d_{\lceil nx\rceil}}{n}\right)^2 K^2(x)}\right|,\\
&\le \frac{8 \left|K(x)-\frac{d_{\lceil nx\rceil}}{n}\right|}{K_{min}^4}.
\end{align*}Using this for every $\delta>0$, we obtain, as $n \to \infty,$

\begin{align*}
&\mathbb{P}\left(\int_0^1\left|\frac{n^2}{d_{\lceil nx\rceil}^2} - \frac{1}{K^2(x)}\right|dx>\delta\right),\\
\le & \mathbb{P}\left(\int_0^1\left|\frac{n^2}{d_{\lceil nx\rceil}^2} - \frac{1}{K^2(x)}\right|dx>\delta \cap E_n\right) + \mathbb{P}(E_n^c),\\
\le & \mathbb{P}\left(\int_0^1\left|\frac{d_{\lceil nx\rceil}}{n}-K(x)\right|dx>\frac{\delta K_{min}^4}{4} \right) + \mathbb{P}(E_n^c) \to 0,\\
\end{align*}where the last line follows using the fact that $\mathbb{P}(E_n^c) \to 0$ and Remark \ref{DegDevL1}.
Using Lemma \ref{LargestEigenvalueConvergence} we conclude that 
\begin{align}
\label{EFPAPart1a}\int_0^1 \frac{\lambda_1^2}{d_{\lceil nx\rceil}^2}dx \overset{p}{\to} \rho^2\int_0^1\frac{1}{K^2(x)}dx.
\end{align}
A similar argument gives for $\delta >0$,
\begin{align*}
& \mathbb{P}\left(\int_0^1 \frac{1}{K(x)}\left|\frac{1}{d_{\lceil nx\rceil}/n}-\frac{1}{K(x)}\right|dx>\delta\right)\\ 
\le & \mathbb{P}\left(\int_0^1\left|\frac{d_{\lceil nx\rceil}}{n}-K(x)\right|dx>\frac{\delta K_{min}^3}{2} \right) + \mathbb{P}(E_n^c)\to 0,
\end{align*}and hence we get
\begin{align}
\label{EFPAPart1b} 2\lambda_1\rho\int_0^1\frac{1}{d_{\lceil nx \rceil}K(x)}dx \overset{p}{\to } 2\rho^2\int_0^1\frac{1}{K^2(x)}dx.
\end{align}
Equations \eqref{EFPAPart1a} and \eqref{EFPAPart1b} now complete the proof for \textbf{Model A}.   
For \textbf{Model B} we define the sequence of events $G_n$ and $H_n$ as 
\[
G_n = \left[\min_i \frac{d_i}{n\varepsilon_n} > \frac{K_{min}}{2}\right],
\]
and
\[
H_n = \left[\max_i \frac{d_i}{n\varepsilon_n} < \frac{3}{2}\right],
\]
The exact same argument for $E_n$ in \textbf{Model A} continues verbatim and gives 
\[
\mathbb{P}(G_n^c) \to 0. 
\]
Similarly, using the fact that $|K|\le 1,$ we get
\[
\max_i \frac{d_i}{n\varepsilon_n} \le 1+ \max_i\left|\frac{d_i}{n\varepsilon_n} - K\left(\frac{i}{n}\right)\right|,
\] and hence,
\begin{align*}
\mathbb{P}(H_n^c) =\mathbb{P}\left[\max_i \frac{d_i}{n\varepsilon_n} \ge \frac{3}{2}\right] \le \mathbb{P}\left(\max_i\left|\frac{d_i}{n\varepsilon_n} - K\left(\frac{i}{n}\right)\right| \ge \frac{1}{2}\right) \to 0,
\end{align*}where the last implication follows again by Lemma \ref{DegMaxDev}.

Now on $G_n \cap H_n,$ we have using $|K|\le 1$ and assumption $A1$,  
\begin{align*}
\left|\frac{(n\varepsilon_n)^2}{d_{\lceil nx\rceil}^2} - \frac{1}{K^2(x)}\right|
&= \left|\frac{\left(K(x)+\frac{d_{\lceil nx\rceil}}{n\varepsilon_n}\right)\left(K(x)-\frac{d_{\lceil nx\rceil}}{n\varepsilon_n}\right)}{\left(\frac{d_{\lceil nx\rceil^2}}{(n\varepsilon_n)}\right)^@K^2(x)}\right|
\le \frac{10 \left|\frac{d_{\lceil nx\rceil}}{n} - K(x)\right|}{K_{min}^4}.
\end{align*}Thus, once again as in the proof of \textbf{Model A}, for $\delta>0$, we have as $n \to \infty,$

\begin{align*}
\mathbb{P}\left(\int_0^1\left|\frac{(n\varepsilon_n)^2}{d_{\lceil nx\rceil}^2} - \frac{1}{K^2(x)}\right|dx>\delta\right)
\le & \mathbb{P}\left(\int_0^1\left|\frac{d_{\lceil nx\rceil}}{n\varepsilon_n}-K(x)\right|dx>\frac{\delta K_{min}^4}{10} \right) + \mathbb{P}(G_n^c \cup H_n^c)\to 0,
\end{align*}where the last line follows using Remark \ref{DegDevL1} and the fact that $\mathbb{P}(G_n^c \cup H_n^c) \to 0$.
Using Lemma \ref{LargestEigenvalueConvergence} again, we conclude that 
\begin{align}
\label{EFPBSPart1a}\int_0^1 \frac{\lambda_1^2}{d_{\lceil nx\rceil}^2}dx \overset{p}{\to} \rho^2\int_0^1\frac{1}{K^2(x)}dx.
\end{align}
A similar argument as in \textbf{Model A} gives for $\delta >0$,
\begin{align*}
\mathbb{P}\left(\int_0^1 \frac{1}{K(x)}\left|\frac{1}{d_{\lceil nx\rceil}/n\varepsilon_n}-\frac{1}{K(x)}\right|dx>\delta\right)\to 0,
\end{align*}and hence
\begin{align}
\label{EFPBPart1b} 2\lambda_1\rho\int_0^1\frac{1}{d_{\lceil nx \rceil}K(x)}dx \overset{p}{\to } 2\rho^2\int_0^1\frac{1}{K^2(x)}dx.
\end{align}
Finally, equations \eqref{EFPBSPart1a} and \eqref{EFPBPart1b} complete the proof for \textbf{Model B}.
\end{proof}

\begin{lemma}\label{Part2}[Proof of Part 2]
Recall the definition of the sequence of step functions $r_n$:
\[
r_n(x)=r_{\lceil nx \rceil}^{(n)}\,,\quad x\in[0,1].
\]
Then as  $n \to \infty$,
\[
r_n\overset{p}{\to} \phi \quad \text{in }L^2[0,1].
\]
\end{lemma}
\begin{proof}
We fix a sample point. Let us denote the random integral operator corresponding to the adjacency matrix $A_n$ by $T_{A_n}$ on $L^2[0,1]$ with kernel 
\[
A_n\left(\lceil nx \rceil, \lceil ny \rceil\right),
\]
similar to Lemma 2.2 in \cite{bottcher2007norms}. Using the operators $S_n^1$ and $S_n^2$ defined in Lemma \ref{LargestEigenvalueConvergence}, we have 
\begin{align}
\label{ForSpectrumEqualityRandom2}S_n^1A_nS_n^2=nT_{A_n}.
\end{align}
Following the analysis done in Lemma \ref{LargestEigenvalueConvergence} and using \eqref{ForSpectrumEqualityRandom2}, we get 
\begin{align*}
& S_n^1A_nS_n^2(S_n^1r)=\lambda_1(S_n^1r),\\
\implies& nT_{A_n}(S_n^1r)=\lambda_1(S_n^1r),
\end{align*}
i.e. $\lambda_1$ is an eigenvalue of the operator $S_n^1A_nS_n^2$ with eigenfunction $S_n^1r.$
By definition of $S_n^1$, we have,$$\frac{S_n^1 r}{\sqrt{n}} = r_n.$$
This implies that $r_n$ is an eigenfunction of $nT_{A_n}$ corresponding to the largest eigenvalue $\lambda_1.$ By definition of eigenvalues and eigenfunctions, $r_n$ continues to be an eigenfunction of $T_{A_n}$ corresponding to its largest eigenvalue $\frac{\lambda_1}{n}$.
Recall Assumption A2, which gives a spectral gap at $\rho=\lambda_1(T_W)$ of at least $\gamma$. This implies, on the event $|\lambda_1(T_{A_n})-\rho|<\gamma/2$, the hypotheses of Lemma 7 in \cite{avella2018centrality}, a version of the Davis--Kahan theorem, are satisfied.\\ 
Therefore, on the event $\left|\lambda_1\left(T_
{A_n}\right)-\lambda_1(T_W)\right| < \frac{\gamma}{2},$ we have 
\begin{align*}
\|r_n -\phi\|_2 &\le \frac{\sqrt{2}\|T_{A_n} - T_W\|}{\gamma -\left|\lambda_1\left(T_
{A_n}\right)-\lambda_1(T_W)\right| },\\
& \le \frac{2\sqrt{2}\|T_{A_n} - T_W\|}{\gamma}.
\end{align*}
Hence, for $\eta >0,$
\begin{align*}
&\mathbb{P}(\|r_n - \phi\|_2 > \eta),\\
\le & \mathbb{P}\left(\|r_n - \phi\|_2 > \eta,\left|\lambda_1\left(T_
{A_n}\right)-\lambda_1(T_W)\right| < \frac{\gamma}{2} \right) + \mathbb{P}\left(\left|\lambda_1\left(T_
{A_n}\right)-\lambda_1(T_W)\right| \ge \frac{\gamma}{2}\right),\\
\le & \mathbb{P}\left(\frac{2\sqrt{2}\|T_{A_n} - T_W\|}{\gamma} > \eta\right) + \mathbb{P}\left(\left|\lambda_1\left(T_
{A_n}\right)-\lambda_1(T_W)\right| \ge \frac{\gamma}{2}\right),\\
\le & \mathbb{P}\left(\frac{2\sqrt{2}\|T_{A_n} - T_W\|}{\gamma} > \eta\right) + \mathbb{P}\left(\left\|T_{A_n} - T_W\right\| \ge \frac{\gamma}{2}\right).
\end{align*}
where the last line follows by Weyl's inequality as $T_{A_n}$ and $T_W$ are self-adjoint compact operators.
Now it remains to show that \[
\left\|T_{A_n} - T_W\right\| = o_p(1).
\]To see this, note that,
\begin{align*}
\left\|T_{A_n} - T_W\right\| \le \left\|T_{A_n} - T_{W_n}\right\|+\left\|T_{W_n} - T_W\right\|.
\end{align*}That the second term  in the rhs is $o_p(1)$ has been proved in Lemma \ref{LargestEigenvalueConvergence}. For the first term, we note that
\begin{align*}
\left\|T_{A_n} - T_{W_n}\right\| 
& = \frac{1}{n}\left\|A_n - P_n\right\|,\\
& = O_p\left(\frac{1}{\sqrt{n}}\right),
\end{align*}where the second last line follows using Lemma 2.2 from \cite{bottcher2007norms} and the last line follows from Lemma \ref{ConcentrationAroundMean}. This completes the proof for \textbf{Model A}. \\

For \textbf{Model B}, the same argument shows that $r_n$ is an eigenfunction of $T_{A_n}/\epsn$ corresponding to the largest eigenvalue $\lambda_1/(n\epsn)$. Applying the Davis--Kahan bound on the event
\[
\left|\lambda_1\left(T_{A_n}/\epsn\right)-\lambda_1(T_W)\right|<\frac{\gamma}{2}
\]
gives
\[
\|r_n-\phi\|_2\le \frac{2\sqrt2}{\gamma}\left\|\frac{T_{A_n}}{\epsn}-T_W\right\|.
\]
Finally,
\[
\left\|\frac{T_{A_n}}{\epsn}-T_W\right\|
\le \frac{1}{n\epsn}\|A_n-P_n\|+\|T_{W_n}-T_W\|=o_p(1),
\]
where the first term is handled by Lemma \ref{ConcentrationAroundMean} and the second by the operator convergence in the proof of Lemma \ref{LargestEigenvalueConvergence}. This completes the proof for \textbf{Model B}.

\end{proof}
\begin{proof}[Proof of \bf Theorem \ref{Th.ESDEFB}]

By Lemmas \ref{Part1} and \ref{Part2}, together with \eqref{MainFormulation}, we have
\[
\int_0^1\left|\Delta_n^I(x)-\left(\frac{\rho}{K(x)}-1\right)\phi(x)\right|\,dx \overset{p}{\to}0,
\]
where
\[
\Delta_n^I(x)=\sum_{i=1}^n \Delta_{i,n}^E\mathbf 1_{[(i-1)/n,i/n)}(x).
\]
Let $F$ be bounded and Lipschitz, with Lipschitz constant $L_F$. Then
\begin{align*}
\left|\frac1n\sum_{i=1}^nF(\Delta_{i,n}^E)-\int_0^1F\left(\left(\frac{\rho}{K(x)}-1\right)\phi(x)\right)dx\right|
&\le L_F\int_0^1\left|\Delta_n^I(x)-\left(\frac{\rho}{K(x)}-1\right)\phi(x)\right|dx\\
&\overset{p}{\to}0.
\end{align*}
This proves \eqref{TargetESDEFB} for every bounded Lipschitz test function. Since bounded Lipschitz functions determine weak convergence of probability measures on $\mathbb R$, the proof is complete for both \textbf{Model A} and \textbf{Model B}.
 
\end{proof}
\begin{remark}\textbf{Some properties of the limiting probability measure:} 
\begin{enumerate}
\item Let us consider the function 
\[
h(x) = \left(\frac{\rho}{K(x)}-1\right)\phi(x).
\]Then $\mu_\infty$ is the law $h(U)$, which means the push-forward of the Lebesgue measure on [0,1] under $h$.
\item Note that as long as $W$ is continuous on $[0,1]^2$, $T_W$ maps $L^2[0,1] \mapsto C[0,1].$ Then, by the definition of eigenvalues and vectors,
\[
\phi(x)=\frac{1}{\rho}\left(T_W \phi\right)(x),
\] $\phi(.)$ is also continuous on [0,1]. 
\item When $K>0$, $h$ is continuous on [0,1]. In that case, $h[0,1]$ is a bounded interval and, hence, in this case $\mu_\infty$ has a compact support.
\end{enumerate}
\end{remark}

\begin{remark}\label{EFP siginicance}\textbf{Significance of the EFP :}
\cite{HdHP2023} has defined the friendship paradox significance by the property 
$\mu_\infty[0,\infty) > 1/2,$ where $\mu_\infty$ in our case is the limiting measure, i.e. the law of $\Delta_\infty= \left(\frac{\rho}{K(U)}-1\right)\phi(U)$, as defined in Theorem \ref{Th.ESDEFB}. Recall from assumption $A3$ that we have $\phi(x) >0$ for all  $x\in[0,1]$. Then,
\begin{align*}
\mathbb{P} [\Delta_\infty \ge 0] & = \mathbb{P}\left(\left(\frac{\rho}{K(U)}-1\right)\phi(U) \ge 0\right)\\
& = \mathbb{P}\left(\frac{\rho}{K(U)}-1\ge 0\right),\\
& = 1 - \mathbb{P}(K(U)>\rho),\\
& \ge 1 - \frac{\int_0^1\int_0^1 W(x,y)dxdy}{\rho},
\end{align*}where the inequality in the second last line follows by Markov's inequality and the fact that $\mathbb{E}(K(U)) = \int_0^1\int_0^1 W(x,y)dxdy$. 
Thus, the EFP significance holds when 
\begin{align}
\label{EFPSignificance1}\int_0^1\int_0^1 W(x,y)dxdy < \rho/2.
\end{align}
\end{remark}

\section{Few Illustrative Examples}
\label{s.examples}
\subsection{d-regular graph:}
Note that when we look at the adjacency matrix of a d-regular graph, denoting the sum vector by $\zeta$, we immediately have,
\[
(A\zeta)_i = \sum_{j=1}^n A_{ij}=d,
\]and hence, 
\[
A\zeta=d\zeta.
\]
Thus $d$ is an eigenvalue of $A$ with $\zeta=(1,1,\ldots,1)$ being the corresponding eigenvector. Further, the fact that 
\[
\|A\| \le \max_i\sum_{j=1}^n A_{ij}=d,
\]we get that $d$ is the largest eigenvalue of $A$, with the normalized eigenvector $\zeta/\sqrt{n}$.
Hence for any vertex $V_i$ we have,
\[
\Delta_{i,n}^E = \frac{\lambda_1}{d_i}r_i-r_i=0,
\]implying that $\Delta_n^E =0,$ which means there is zero eigen bias and thus,\[\mu_\infty = \delta_{0}.\]
Clearly \[\mu_\infty[0,\infty) = 1 >1/2,\]
and hence the EFP significance holds trivially in this case.
\subsection{Rank 1 Graphon:} For the case when we have a rank-1 graphon i.e.,
\[
W(x,y) = r(x) r(y),
\]we have,
\[
K(x)=r(x)\int_0^1r(y)dy.
\]
By Assumption $A3$ theorem, we know that, $\rho >0$ and $\phi(x) >0$ for all $x\in [0,1].$ Then, from the eigenvalue equation we get, for all $x \in [0,1],$
\begin{align*}
\int_0^1W(x,y)\phi(y)dy &=\rho\phi(x),\\
r(x)\int_0^1r(y)\phi(y)dy&= \rho\phi(x),\\
\phi(x)&=Cr(x),
\end{align*}where 
\[
C = \frac{1}{\rho}\int_0^1r(y)\phi(y)dy.
\]Clearly then $C>0$ as $\phi(x) >0$ for all $x\in [0,1].$
Using Assumption $A3,$ we get 
\[
C= \frac{1}{\|r\|_2}.
\]
This implies that $\phi(x) = \frac{r(x)}{\|r\|_2}.$
Therefore, using the eigenvalue equation again, we get for all $x \in [0,1],$
\begin{align*}
C\int_0^1W(x,y)r(y)dy &=C\rho r(x),\\
\text{which implies}\,\,r(x)\int_0^1r^2(y)dy& = \rho r(x),\\
\text{and hence}\,\, \rho &= \int_0^1r^2(y)dy.
\end{align*}
\[
\Delta_\infty = \left[\frac{\int_0^1r^2(y)dy}{r(.)\int_0^1r(y)dy}-1\right]\frac{r(.)}{\|r\|_2}.
\]
Thus,
\begin{align*}
\mathbb{E}\Delta_\infty &= \frac{\int_0^1r^2(x)dx}{\|r\|_2\int_0^1r(x)dx}-\frac{\int_0^1r(x)dx}{\|r\|_2},\\
& = \frac{\int_0^1r^2(x)dx-\left(\int_0^1 r(x)dx\right)^2}{\|r\|_2\int_0^1r(y)dy}\ge 0,
\end{align*}where the last inequality follows from Cauchy-Schwarz inequality. In view of \eqref{EFPSignificance1} we get that the EFP significance holds for the rank-1 graphon when,
\begin{align}
\label{EFPSignificanceRank1} \int_0^1 r^2(x)dx  > 2\left(\int_0^1 r(x)dx\right)^2.   
\end{align}

\subsection{A two block graphon:}Let us consider the two block graphon defined as,
\[
W(x,y) =
\begin{cases}
P_{11}, & \mbox{if } (x,y) \in [0,\alpha] \times [0,\alpha],\\
P_{12}, & \mbox{if } (x,y) \in [0,\alpha] \times [\alpha,1] \cup [\alpha,1] \times [0,\alpha],\\
P_{22}, & \mbox{if } (x,y) \in [\alpha,1] \times [\alpha,1],
\end{cases}
\]
with $\alpha \in (0,1)$ and $P_{11},P_{22} \in [0,1]$, and $P_{12} \in (0,1]$ as \vrdg{summarized in} Figure \ref{fig:twoblockgraphon}.

\vspace{0.3cm}
\begin{figure}[htbp]
\centering
\begin{tikzpicture}[scale=0.7]
\draw (0,0)--(0,4)--(4,4)--(4,0)--(0,0);
\draw (1.5,0)--(1.5,4);
\draw (0,1.5)--(4,1.5);
\draw [<->] (0,-0.3) -- (1.5,-0.3) node[midway,below]{$\alpha$};
\draw [<->] (1.5,-0.3)-- (4,-0.3) node[midway,below]{$1-\alpha$};
\node at (0.75,0.75) {$P_{11}$};
\node at (2.75,2.75) {$P_{22}$};
\node at (0.75,2.75) {$P_{12}$};
\node at (2.75,0.75) {$P_{12}$};
\end{tikzpicture}
\caption{\small A two-block graphon.}
\label{fig:twoblockgraphon}
\end{figure}
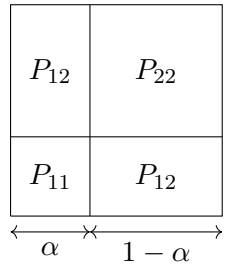

Although this graphon is not continuous at the block boundary, the following finite-type computation is useful and can be justified by continuous approximations. Here, the degree functional takes a simplified $2-$valued form given by,
\[
K(x)=
\begin{cases}
P_{11}\alpha + P_{12}(1-\alpha), & \mbox{if } x \in [0,\alpha],\\
P_{12}\alpha + P_{22}(1-\alpha), & \mbox{if } x \in (\alpha,1].\\
\end{cases}
\]

In what follows, we will denote the above values of the degree functional by $K_1$ and $K_2.$ 
Since our graphon has the 2-block structure, it follows that the eigenfunction $\phi$ is also 2-valued, of the form, say
\begin{align}
\label{2blockEVecImplicit}\phi(x)=
\begin{cases}
u, & \mbox{if } x \in [0,\alpha],\\
v, & \mbox{if } x \in (\alpha,1].\\
\end{cases}    
\end{align}
From definition of $T_W$ we get, for $0\le x\le \alpha,$
\begin{align}
\label{2block1} \alpha P_{11}u + P_{12}(1-\alpha)v = \rho u,  
\end{align}and for $\alpha < x\le 1,$
\begin{align}
\label{2block2} \alpha P_{12}u + P_{22}(1-\alpha)v = \rho v.  
\end{align}
 A direct computation leads to,
\begin{align}
\label{2blockEValMax}\rho = \frac{\alpha P_{11}+(1-\alpha)P_{22}+\sqrt{(\alpha P_{11}-(1-\alpha)P_{22})^2+4\alpha(1-\alpha)P_{12}^2}}{2},
\end{align}and
\begin{align}
\label{2blockEVecMax}\phi(x) =
\frac{1}{\sqrt{\alpha+(1-\alpha)\left(\frac{\rho-\alpha P_{11}}{(1-\alpha)P_{12}}\right)^2}}\begin{cases}
1, & \mbox{if } x \in [0,\alpha],\\
\frac{\rho-\alpha P_{11}}{(1-\alpha)P_{12}}, & \mbox{if } x \in (\alpha,1].\\
\end{cases}
\end{align}
Using the implicit $2-$valued form for $\phi$ stated in \eqref{2blockEVecImplicit}, the expected eigenfriendship bias can be written as
\begin{align}
\label{2blockEEFB1}\int_0^1\left(\frac{\rho}{K(x)}-1\right)\phi(x)dx = 
\frac{\rho u-K_1u}{K_1}\alpha + \frac{\rho v-K_2v}{K_2}(1-\alpha).
\end{align}
Equation \eqref{2block1} implies that 
\begin{align*}
\rho u-K_1 u = (1-\alpha)P_{12}(v-u),
\end{align*}and equation \eqref{2block2} implies that,
\begin{align*}
\rho v-K_2 v=\alpha P_{12}(u-v).    
\end{align*}Thus \eqref{2blockEEFB1} simplifies as 
\begin{align}
\label{2blockEEFB2}\int_0^1\left(\frac{\rho}{K(x)}-1\right)\phi(x)dx = 
\alpha(1-\alpha)P_{12}(v-u)\left[\frac{K_2 - K_1}{K_1K_2}\right].
\end{align}If $v=u$ or $K_1=K_2$ then the above is 0, so we exclude those cases. The non-negativity of the EEFB will follow once we show that \[
\operatorname{sgn}(v-u)=\operatorname{sgn}(K_2-K_1).
\]
Subtracting \eqref{2block1} from \eqref{2block2} we get,
\begin{align*}
\rho(v-u)=v\left((1-\alpha)P_{22}-(1-\alpha)P_{12}\right) + u(\alpha P_{12}-\alpha P_{11}).
\end{align*}Substituting the values of $\alpha P_{11}$ and $(1-\alpha)P_{22}$ from the expressions for $K_1$ and $K_2$ we get,
\begin{align}
\nonumber\rho(v-u) &= v(K_2 - P_{12}) + u(P_{12}-K_1),\\ 
\nonumber &= vK_2-uK_1-P_{12}(v-u),\\
\nonumber \implies (\rho + P_{12})(v-u) &= K_2(v-u)+u(K_2-K_1),\\
\label{2blockEEFBsign}\implies (\rho + P_{12}-K_2)(v-u)&=u(K_2-K_1),
\end{align}thus the non-negativity of the EEFB will follow once it is shown that $\rho+P_{12}-K_2>0.$ 
To see this, from \eqref{2block2} we get,
\begin{align*}
\rho = \alpha P_{12}u/v + (1-\alpha) P_{22},   
\end{align*}and thus,
\begin{align*}
\rho + P_{12}-K_2 &= \alpha P_{12}u/v + (1-\alpha)P_{22}+ P_{12}-K_2,\\
& = \alpha P_{12}u/v + P_{12}(1-\alpha)>0,
\end{align*}
which is strictly positive since $\alpha\in(0,1)$, $P_{12}>0$, and $u/v>0$.

\bibliographystyle{apalike}
\bibliography{Bibliography}

\end{document}